\documentclass[11pt]{article}
\usepackage{float}
\usepackage{amssymb}
\usepackage{amsmath}
\usepackage{graphicx}
\usepackage{color}
\usepackage{ifpdf}
\newtheorem{theorem}{Theorem}[section]

\newtheorem{corollary}[theorem]{Corollary}

\newtheorem{definition}[theorem]{Definition}

\newtheorem{lemma}[theorem]{Lemma}

\newenvironment{proof}[1][Proof]{\noindent \textbf{#1.} }{\hfill  \rule{0.5em}{0.5em}}
\def \v{\mathbf{v}}
\def \u{\mathbf{u}}
\def \x{\mathbf{x}}
\def \A{\mathcal{A}}
\def \L{\mathcal{L}}
\def \Q{\mathcal{Q}}
\def \T{\mathcal{T}}
\def \Gn{\mathcal{G}_n}
\def \Gnb{^{\rm nb}\!\mathcal{G}_n}
\def \Gnk{\mathcal{G}_n^{k,\frac{k}{2}}}
\def \Gnobk{^{\rm nob}\!\mathcal{G}_n^{k,\frac{k}{2}}}
\def \Gk{\mathcal{G}^{k,\frac{k}{2}}}
\def \Gobk{^{\rm nob}\!\mathcal{G}^{k,\frac{k}{2}}}
\begin{document}
\title{\bf On the spectral radius of a class of non-odd-bipartite even uniform hypergraphs\thanks{Supported by National Natural Science Foundation of China (11371028),
Program for New Century Excellent Talents in University (NCET-10-0001),
Scientific Research Fund for Fostering Distinguished Young Scholars of Anhui University(KJJQ1001),
Academic Innovation Team of Anhui University Project (KJTD001B). }}
\author{Murad-ul-Islam Khan, Yi-Zheng Fan\thanks{Corresponding author.
 E-mail addresses: fanyz@ahu.edu.cn(Y.-Z. Fan), muradulislam@foxmail.com (M. Khan)}\\
    {\small \it School of Mathematical Sciences, Anhui University, Hefei 230601, P. R. China} \\
 }
\date{}
\maketitle

{\bf Abstract}: In order to investigate the non-odd-bipartiteness of even uniform hypergraphs,
starting from a simple graph $G$, we construct a generalized power of $G$, denoted by $G^{k,s}$, which is obtained from $G$ by blowing up each vertex into a $k$-set and each edge
into a $(k-2s)$-set, where $s \le k/2$.
When $s < k/2$,   $G^{k,s}$ is always odd-bipartite.
We show that $G^{k,{k \over 2}}$ is non-odd-bipartite if and only if $G$ is non-bipartite, and find that $G^{k,{k \over 2}}$ has the same adjacency (respectively, signless Laplacian) spectral radius as $G$.
So the results involving the adjacency or signless Laplacian spectral radius of a simple graph $G$ hold for  $G^{k,{k \over 2}}$.
In particular, we characterize the unique graph with minimum adjacency or signless Laplacian spectral radius among all non-odd-bipartite hypergraphs $G^{k,{k \over 2}}$ of fixed order, and
prove that $\sqrt{2+\sqrt{5}}$ is the smallest limit point of the non-odd-bipartite hypergraphs $G^{k,{k \over 2}}$.
In addition we obtain some results for the spectral radii of the weakly irreducible nonnegative tensors.

%



\section{Introduction}

Hypergraphs are a generalization of simple graphs.
They are really handy to show complex relationships found in the real world.
A {\it hypergraph} $G=(V(G),E(G))$ is a set of vertices say $V(G)=\{v_1,v_2,\ldots,v_n\}$ and a set of edges, say $E(G)=\{e_{1},e_2,\ldots,e_{m}\}$ where $e_{j}\subseteq V(G)$.
If $|e_{j}|=k$ for each $j=1,2,\ldots,m$, then $G$ is called a {\it $k$-uniform} hypergraph.
In particular, the $2$-uniform hypergraphs are exactly the classical simple graphs.
The {\it degree} $d_v$ of a vertex $v \in V(G)$ is defined as $d_v=|\{e_{j}:v\in e_{j}\in E(G)\}|$.
 A {\it walk} $W$ of length $l$ in $G$ is a sequences of alternate vertices and edges: $v_{0},e_{1},v_{1},e_{2},\ldots,e_{l},v_{l}$,
    where $\{v_{i},v_{i+1}\}\subseteq e_{i}$ for $i=0,1,\ldots,l-1$.
If $v_0=v_l$, then $W$ is called a {\it circuit}.
A walk in $G$ is called a {\it path} if no vertices or edges are repeated.
A circuit in $G$  is called a {\it cycle} if no vertices or edges are repeated.
The hypergraph $G$  is said to be {\it connected} if every two vertices are connected by a walk.
A hypergraph $H$ is a {\it sub-hypergraph} of $G$ if $V(H)\subseteq V(G)$ and $E(H) \subseteq E(G)$,
and $H$ is a {\it proper sub-hypergraph} of $G$ if $V(H)\subsetneq V(G)$ or $E(H) \subsetneq E(G)$.

In recent years spectral hypergraph theory  has emerged as an important field in algebraic graph theory.
Let $G$ be a $k$-uniform hypergraph.
The {\it adjacency tensor} $\mathcal{A}=\mathcal{A}(G)=(a_{i_{1}i_{2}\ldots i_{k}})$ of $G$ is a $k$th order $n$-dimensional symmetric tensor,
  where
  $$
  a_{i_{1}i_{2}\ldots i_{k}}=\left\{
  \begin{array}{cl}
   \frac{1}{(k-1)!} & \hbox{if~} \{v_{i_{1}},v_{i_{2}},\ldots,v_{i_{k}}\} \in E(G);\\
  0 &  \hbox{otherwise.}
  \end{array}
    \right.
  $$
   Let $\mathcal{D}=\mathcal{D}(G)$ be a $k$th order $n$-dimensional diagonal tensor,
    where $d_{i\ldots i}=d_{v_i}$ for all $i \in [n]:=\{1,2,...,n\}$.
Then $\L=\L(G)=\mathcal{D}(G)-\A(G)$ is the {\it Laplacian tensor} of the hypergraph $G$,
and $\Q=\Q(G)=\mathcal{D}(G)+\A(G)$ is the {\it signless Laplacian tensor} of $G$.

Qi \cite{Qi} showed that $\rho(\L(G)) \le \rho(\Q(G))$, and posed a question of identifying the conditions under which the equality holds.
Hu et al. \cite{HQX} proved that
if $G$ is connected, then the equality holds if and only if $k$ is even and $G$ is odd-bipartite.
Here an  even uniform hypergraph $G$ is called {\it odd-bipartite} if $V(G)$ has a bipartition $V(G)=V_{1}\cup V_{2} $
   such that each edge has an odd number of vertices in both $V_{1}$ and $V_{2}$.
    Such partition will be called an {\it odd-bipartition} of $G$.
Shao et al. \cite{SSW} proved a stronger result that the Laplacian $H$-spectrum (respectively, Laplacian spectrum) and signless Laplacian $H$-spectrum (respectively, Laplacian spectrum) of a connected $k$-uniform hypergraph $G$ are equal
if and only if $k$ is even and $G$ is odd-bipartite.
They also proved that the adjacency $H$-spectrum of $G$ (respectively, adjacency spectrum) is symmetric with respect to the origin if and only if $k$ is even and $G$ is odd-bipartite.
So, the non-odd-bipartite even uniform hypergraphs are more interesting on distinguishing the Laplacian spectrum and signless Laplacian spectrum and studying the non-symmetric adjacency spectrum.

Hu, Qi and Shao \cite{HQS} introduced the {\it cored hypergraphs} and the {\it power hypergraphs},
   where the cored hypergraph is one such that each edge contains at least one vertex of degree $1$,
   and the $k$-th power of a simple graph $G$, denoted by $G^k$, is obtained by replacing each edge (a $2$-set) with a $k$-set by adding $k-2$ new vertices.
These two kinds of hypergraphs are both odd-bipartite.

Peng \cite{P} introduced $s$-path and $s$-cycle.
Suppose $1\le s\leq k-1.$
An {\it $s$-path} $P$ of length $d$ is a $k$-uniform hypergraph on $s+d(k-s)$ vertices, say $v_1,v_2, \ldots,v_{s+d(k-s)}$, such that
$\{v_{1+j(k-s)}$, $v_{2+j(k-s)}$, $\ldots$, $v_{s+(j+1)(k-s)}\}$ is an edge of $P$ for $j=0,\ldots,d-1$.
An {\it $s$-cycle} $C$ of length $d$  is a $k$-uniform hypergraph on $d(k-s)$ vertices, say $v_1,v_2, \ldots,v_{d(k-s)}$, such that
$\{v_{1+j(k-s)}$, $v_{2+j(k-s)}$, $\ldots$, $v_{s+(j+1)(k-s)}\}$ is an edge of $C$ for $j=0,...,d-1$, where $v_{d(k-s)+j}=v_{j}$ for $j=1,...,s$.
When $1\leq s<\frac{k}{2}$, an $s$-path or $s$-cycle is a cored hypergraph and hence it is odd-bipartite.

Up to now, the construction of non-odd-bipartite hypergraphs has rarely appeared.
In Section 2 we proved that an $s$-path is always odd-bipartite. But this does not hold for $s$-cycles.
However, when $s=k/2$ for $k$ being even, an $s$-cycle is odd-bipartite if and only if its length is even, which is consistent with the result on the bipartiteness of a simple cycle. Motivated by the discussion of $s$-cycles, we introduce a class of $k$-uniform hypergraphs, which is obtained from a simple graph by blowing up vertices and/or edges.

\begin{definition}
Let $G=(V,E)$ be a simple graph. For any $k \ge 3$ and $1 \le s \le k/2$, the generalized power of $G$, denoted by $G^{k,s}$, is defined as
the $k$-uniform hypergraph with the vertex set $\{\v: v \in V\} \cup \{\mathbf{e}: e \in E\}$, and the edge set
$\{\u \cup \v \cup \mathbf{e}: e=\{u,v\} \in E\}$, where $\v$ is an $s$-set containing $v$ and $\mathbf{e}$ is a $(k-2s)$-set corresponding to $e$.
\end{definition}

\begin{center}
\includegraphics[scale=.8]{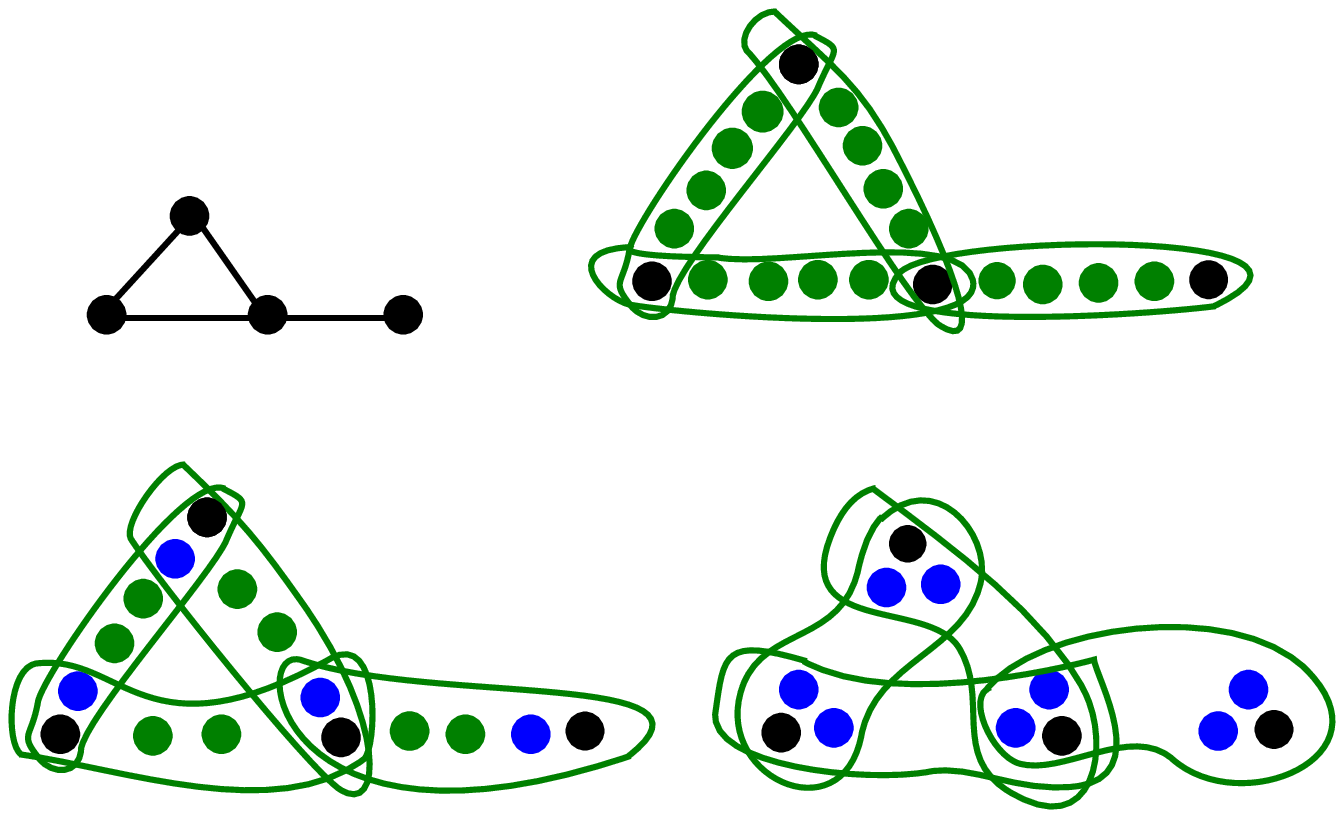}

\small{Fig. 1.1 Constructing power hypergraphs $G^6$ (right upper), $G^{6,2}$ (left below) and $G^{6,3}$ (right below) from a simple graph $G$ (left upper), where a closed curve represents an edge}
\end{center}

Intuitively, $G^{k,s}$ is obtained from $G$ by replacing each vertex $v$ by an $s$-subset $\v$ and each edge $\{u,v\}$ by a $k$-set obtained from $\u \cup \v$ by adding  $(k-2s)$ new vertices; see Fig. 1.1 for illustration.
  If $s=1$, then $G^{k,s}$  is exactly the $k$-th power hypergraph of $G$.
 When $G$ is a path or a cycle, then $G^{k,s}$ is an $s$-path or $s$-cycle for $s \le k/2$.
 So the notion $G^{k,s}$ is a generalization of the above hypergraphs.

Note that if $s<k/2$, then $G^{k,s}$ is a cored hypergraphs and hence is odd-bipartite.
If $s=k/2$, then $G^{k,s}$  is obtained from $G$ by only blowing up its vertices.
In this case, $\{u,v\}$ is an edge of $G$ if and only $\u\cup\v$ is an edge of $G^{k,{k \over 2} }$, where we use the black font $\v$ to denote the blowing-up of the vertex $v$ in $G$.
For simplicity, we write $uv$ rather than $\{u,v\}$, $\u\v$ rather than $\u\cup\v$, and call $\u$ a {\it half edge} of $G^{k,{k \over 2} }$.
In Section 2, we show that $G^{k,{k \over 2} }$ is non-odd-bipartite if and only if $G$ is non-bipartite.
So, we here give an explicit construction of  non-odd-bipartite hypergraphs.

Another problem is how to apply the spectral theory of simple graphs to that of hypergraphs.
In Section 3, we find that $G^{k,{k \over 2}}$ has the same adjacency (respectively, signless Laplacian) spectral radius as $G$.
So the results involving the adjacency or signless Laplacian spectral radius of a simple graph $G$ hold for  $G^{k,{k \over 2}}$.
Here we concern two problems: the minimum adjacency or signless Laplacian spectral radius and the smallest limit point of the graphs  $G^{k,{k \over 2}}$,
which are addressed in Section 4 respectively.

In the paper \cite{LM} the authors proved that the smallest limit point of the adjacency spectral radii of the connected $k$-uniform hypergraphs is $\rho_k=(k-1)!\sqrt[k]{4}$.
(Note that if using our definition for the adjacency tensor, the limit point would be $\sqrt[k]{4}$.)
They also classified all connected $k$-uniform hypergraphs with spectral radii at most $\rho_k$, which are all cored hypergraphs for $k \ge 5$.
(They used the notion of ``reducible hypergraphs'' instead of cored hypergraphs.)
 Even for $k=4$, those graphs are not cored hypergraphs but still odd-bipartite hypergraphs.
So, the next problem is to investigate the smallest limit point of the adjacency spectral radii of the connected $k$-uniform non-odd-bipartite hypergraphs.
We start this problem by considering the class of hypergraphs $G^{k,{k \over 2} }$ where $G$ is non-bipartite.

It is known that a uniform hypergraph is connected if and only if its adjacency tensor is weakly irreducible.
There are many results on the spectral theory of irreducible or weakly irreducible nonnegative tensor, e.g. \cite{CPZ,FGH,YY,YY2,YY3}.
However, to investigate the spectral radius of the adjacency tensor (or signless Laplacian tensor), we still need more results on the weakly irreducible nonnegative tensors.
This will be discussed in Section 3.


\section{Odd-bipartiteness of hypergraphs}
We first discuss the odd-bipartiteness of $s$-paths and $s$-cycles.

\begin{lemma}
An $s$-path is always odd-bipartite where $\frac{k}{2}\leq s\leq k-1.$
\end{lemma}

\begin{proof}
Let $P$ be an $s$-path of length $d$.
If $d=1$, the assertion holds clearly.
Assume the assertion holds for all $s$-paths of length $d < m$.
We prove it by induction on the length.
Consider an $s$-path $P$ of length $m$.
Let $e_{m}$ be the last edge of $P$.
Note that $P-e_m$ is an $s$-path, say $P'$ of length $m-1$, together with $k-s$ isolated vertices.
By induction, $P'$ is odd-bipartite, which has an odd-bipartition $V(P')=V_1 \cup V_2$.
 Now, if $|V_{1}\cap e_{m}|$ is odd, put all vertices of $e_m\backslash V(P')$ into $V_2$.
 Otherwise, take one vertex from $e_m\backslash V(P')$ and put it into $V_1$, and put the remaining into $V_2$.
 Then we get an odd-bipartition of $P$.
\end{proof}

What about the odd-bipartiteness of $s$-cycles when $\frac{k}{2} \leq s\leq k-1$?
We first discuss the case of $s=\frac{k}{2}$. In this case, we use the notation $C_m^{k,{k \over 2}}$ instead, where $C_m$ denote a simple cycle of length $m$.

\begin{lemma}\label{OB-Hypercycle}
The cycle $C_m^{k,{k \over 2}}$ is odd-bipartite if and only if $m$ is even.
\end{lemma}

\begin{proof}
Let $C:=C_m^{k,{k \over 2}}$.
We have a partition of $V(C)=V_1 \cup V_2 \cup \cdots \cup V_m$ such that
$e_i:=V_i \cup V_{i+1}$ is an edge of $C$ for $i=1,2,\ldots,m$, where $V_{m+1}=V_1$.
Suppose that $C$ is  odd-bipartite, which has an odd-bipartition.
We color the vertices in one part of the bipartition with red, and color the vertices in the other part with blue.
Note that $e_1=V_1 \cup V_2$ contains an odd number of red vertices.
Without loss of generality, $V_1$ contains an odd number of red vertices.
So $V_2$ contains an even number of red vertices, and then $V_3$ contains an odd number of red vertices by considering the edge $e_2$.
Repeating  the above discussion, we get that $V_m$ contains an odd number of red vertices if $m$ is odd, and even number of red vertices otherwise.
However, if $m$ is odd, then the edge $e_m=V_m \cup V_1$ would contain an even number of red vertices, a contradiction.
So $m$ is necessarily even.
On the other hand, if $m$ is even, it is easy to give an odd-bipartition of $C$.
\end{proof}

For general case, it may not be easy to determine under which conditions an $s$-cycle is odd-bipartite when $\frac{k}{2}<s\leq k-1$.
For example, let $k=4$, a $3$-cycle of length $8$ is odd-bipartite, but a $3$-cycle of length $6$ is non-odd-bipartite.
We will not investigate this problem further in this paper.
By Lemma \ref{OB-Hypercycle}, $C_m^{k,{k \over 2}}$ is non-odd-bipartite if and only if $C_m$ is non-bipartite.
We generalize this fact as follows.

\begin{theorem}\label{NOB}
The hypergraph $G^{k,{k \over 2}}$ is non-odd-bipartite if and only if $G$ is non-bipartite.
\end{theorem}

\begin{proof}
We prove an equivalent assertion: $G^{k,{k \over 2}}$ is odd-bipartite if and only if $G$ is bipartite.
Assume that $G^{k,{k \over 2}}$ is odd-bipartite.
If $G$ is a forest, surely it is bipartite.
Otherwise, any cycle of $G^{k,{k \over 2}}$ must has the form $C_m^{k,{k \over 2}}$ for some positive integer $m$.
 Then $C_m^{k,{k \over 2}}$ is also odd-bipartite and hence $C_m$ is bipartite by Lemma \ref{OB-Hypercycle}.
So, $G$ is bipartite.

On the contrary, assume that $G$ is bipartite, with a bipartition $(V_1,V_2)$.
Extend this bipartition to a bipartition $(\mathbf{V}_1,\mathbf{V}_2)$ of $G^{k,{k \over 2}}$,
  that is, $\mathbf{V}_1$ (respectively, $\mathbf{V}_2$) is obtained by replacing each vertex in $V_1$ (respectively $V_2$) by the corresponding half edge.
Choosing an arbitrary vertex from each half edge in $\mathbf{V}_1$ and forming a new set $U_1$,
then $(U_1, V(G^{k,{k \over 2}})\backslash U_1)$ is an odd-bipartition of $G^{k,{k \over 2}}$.
\end{proof}

\section{Spectral radii and eigenvectors of hypergraphs}

For integers $k\geq 3$ and $n\geq 2$,
  a real {\it tensor} (also called {\it hypermatrix}) $\mathcal{T}=(t_{i_{1}\ldots i_{k}})$ of order $k$ and dimension $n$ refers to a
  multidimensional array with entries $t_{i_{1}\ldots i_{k}}$ such that $t_{i_{1}\ldots i_{k}}\in \mathbb{R}$ for all $i_{j}\in [n]$ and $j\in [k]$.
 The tensor $\mathcal{T}$ is called \textit{symmetric} if its entries are invariant under any permutation of their indices.
 Given a vector $x\in \mathbb{R}^{n}$, $\mathcal{T}x^{k}$ is a real number, and $\mathcal{T}x^{k-1}$ is an $n$-dimensional vector, which are defined as follows:
   $$\mathcal{T}x^{k}=\sum_{i_1,i_{2},\ldots,i_{k}\in [n]}t_{i_1i_{2}\ldots i_{k}}x_{i_1}x_{i_{2}}\cdots x_{i_k},~
   (\mathcal{T}x^{k-1})_i=\sum_{i_{2},\ldots,i_{k}\in [n]}t_{ii_{2}\ldots i_{k}}x_{i_{2}}\cdots x_{i_k} \mbox{~for~} i \in [n].$$
 Let $\mathcal{I}$ be the {\it identity tensor} of order $k$ and dimension $n$, that is, $i_{i_{1}i_2 \ldots i_{k}}=1$ if and only if
   $i_{1}=i_2=\cdots=i_{k} \in [n]$ and zero otherwise.

\begin{definition}{\em \cite{Qi2}} Let $\mathcal{T}$ be a $k$th order n-dimensional real tensor.
For some $\lambda \in \mathbb{C}$, if the polynomial system $(\lambda \mathcal{I}-\mathcal{T})x^{k-1}=0$, or equivalently $\mathcal{T}x^{k-1}=\lambda x^{[k-1]}$, has a solution $x\in \mathbb{C}^{n}\backslash \{0\}$,
then $\lambda $ is called an eigenvalue of $\mathcal{T}$ and $x$ is an eigenvector of $\mathcal{T}$ associated with $\lambda$,
where $x^{[k-1]}:=(x_1^{k-1}, x_2^{k-1},\ldots,x_n^{k-1}) \in \mathbb{C}^n$.
\end{definition}

If $x$ is a real eigenvector of $\mathcal{T}$, surely the corresponding eigenvalue $\lambda$ is real.
In this case, $x$ is called an {\it $H$-eigenvector} and $\lambda$ is called an {\it $H$-eigenvalue}.
Furthermore, if $x\in \mathbb{R}_{+}^{n}$ (the set of nonnegative vectors of dimension $n$), then $\lambda $ is called an {\it $H^{+}$-eigenvalue} of $\mathcal{T}$;
if $x\in \mathbb{R}_{++}^{n}$ (the set of positive vectors of dimension $n$), then $\lambda$ is said to be an {\it $H^{++}$-eigenvalue} of $\mathcal{T}$.
The {\it spectral radius of $\T$} is defined as
$$\rho(\T)=\max\{|\lambda|: \lambda \mbox{ is an eigenvalue of } \T \}.$$

To generalize the classical Perron-Frobenius Theorem from nonnegative matrices to nonnegative tensors, we need the definition of the irreducibility of tensor.
Chang et al. \cite{CPZ} introduced the irreducibility of tensor. A tensor $\T=(t_{i_{1}...i_{k}})$ of order $k$ and dimension $n$ is called {\it reducible} if there exists a nonempty proper subset $I \subsetneq [n]$ such that
$t_{i_1i_2\ldots i_k}=0$ for any $i_1 \in I$ and any $i_2,\ldots,i_k \notin I$.
If $\T$ is not reducible, then it is called {\it irreducible}.

Friedland et al. \cite{FGH} proposed a weak version of irreducible nonnegative tensors $\T$.
The graph associated with $\T$, denoted by $G(\T)$, is the directed graph with vertices $1, \ldots, n$ and an edge from $i$ to $j$
 if and only if $t_{i_1i_2\ldots i_k}>0$ for some $i_l = j$, $l = 2, \ldots, m$.
The tensor $\T$ is called {\it weakly irreducible} if $G(\T)$ is strongly connected.
Surely, an irreducible tensor is always weakly irreducible.
Pearson and Zhang \cite{PZ} proved that the adjacency tensor of a uniform hypergraph $G$ is weakly irreducible if and only if $G$ is connected.
Clearly, this shows that if $G$ is connected, then $\mathcal{A}(G), \mathcal{L}(G)$ and $\mathcal{Q}(G)$ are all weakly irreducible.

\begin{theorem}\label{PF} {\em \bf (The Perron-Frobenius Theorem for Nonnegative Tensors)}

1. {\em (Yang and Yang 2010 \cite{YY})} If $\T$ is a nonnegative tensor of order $k$ and dimension $n$, then $\rho(\T)$ is an $H^+$-eigenvalue of $\T$.

2. {\em (Frieland, Gaubert and Han 2011 \cite{FGH})} If furthermore $\T$ is weakly irreducible, then $\rho(\T)$ is the unique $H^{++}$-eigenvalue of $\T$,
with the unique eigenvector $x \in \mathbb{R}_{++}^{n}$, up to a positive scaling coefficient.

3. {\em(Chang, Pearson and Zhang 2008 \cite{CPZ})} If moreover $\T$ is irreducible, then $\rho(\T)$ is the unique $H^{+}$-eigenvalue of $\T$,
with the unique eigenvector $x \in \mathbb{R}_{+}^{n}$, up to a positive scaling coefficient.

\end{theorem}

\begin{theorem} \label{ineq1}{\em \cite{YY,YY3}}
Let $\mathcal{B},\mathcal{C}$ be order $k$ dimension $n$ tensors satisfying $|\mathcal{B}| \le \mathcal{C}$, where  $\mathcal{C}$ is weakly irreducible.
Let $\beta$ be an eigenvalue of $\mathcal{B}$. Then

(1) $|\beta| \le \rho(\mathcal{C})$.

(2) if $\beta=\rho(\mathcal{C})e^{i \varphi}$ and $y$ is corresponding eigenvector, then
all entries of $y$ are nonzero, and $\mathcal{C}=e^{-i \varphi}\mathcal{B} \cdot D^{-(k-1)} \cdot \overbrace{ D \cdots D }^{k-1}$,
where $D=\hbox{diag}(\frac{y_1}{|y_1|},\frac{y_2}{|y_2|},\ldots,\frac{y_n}{|y_n|})$.

\end{theorem}

\begin{corollary}\label{ineq2}
Suppose $0 \le \mathcal{B} \lneq \mathcal{C}$, where  $\mathcal{C}$ is weakly irreducible.
Then $\rho(\mathcal{B}) < \rho(\mathcal{C})$.
\end{corollary}

\begin{proof}
By Theorem \ref{PF}(1), $\rho(\mathcal{B})$ is an eigenvalue of $\mathcal{B}$, with a nonnegative eigenvector $y$.
By Theorem \ref{ineq1}, $\rho(\mathcal{B}) \le \rho(\mathcal{C})$.
If $\rho(\mathcal{B}) = \rho(\mathcal{C})$, then, also by Theorem \ref{ineq1}, $y>0$, and hence $\mathcal{B} = \mathcal{C}$; a contradiction.
\end{proof}

\begin{corollary}\label{ineq3}
Suppose $G$ is a connected $k$-uniform hypergraph and $H$ is a proper sub-hypergraph of $G$.
Then $\rho(\mathcal{A}(H)) < \rho(\mathcal{A}(G))$ and $\rho(\mathcal{Q}(H)) < \rho(\mathcal{Q}(G))$.
\end{corollary}

\begin{proof}
We only consider the adjacency tensor. The other case can be discussed in a similar manner.
Observe that $\mathcal{A}(G)$ is weakly irreducible.
First assume $V(H)=V(G)$.
Then  $\mathcal{A}(H)) \lneq \mathcal{A}(G)$, which implies the result by Corollary \ref{ineq2}.
Secondly assume $V(H) \subsetneq V(G)$.
Add the isolated vertices of $V(G) \backslash V(H)$ to $H$ such that the resulting hypergraph, say $H'$, has the same vertex set as $G$.
Then $\rho(\mathcal{A}(H)) = \rho(\mathcal{A}(H'))$; or see \cite[Theroem 3.2]{ZQW}.
Noting that $\mathcal{A}(H')) \lneq \mathcal{A}(G)$, we also get the result.
\end{proof}

Let $\mathcal{B} \ge 0$. Let $x \in \mathbb{R}^n_{++}$.
Denote $$r_i(\mathcal{B})=\sum_{i_2,\ldots,i_k=1}^n b_{ii_2 \ldots i_k}, \ s_i(\mathcal{B},x)=\frac{(\mathcal{B}x^{k-1})_i}{x_i^{k-1}},
 \hbox{~for~} i=1,2,\ldots,n.$$
The following two results give bounds for the spectral radius $\rho(\mathcal{B})$ of a general nonnegative tensor $\mathcal{B}$.
Here we impose an additional condition on $\mathcal{B}$, that is, $\mathcal{B}$ is weakly irreducible, and characterize the equality cases.

\begin{lemma}{\em \cite[Lemma 5.2]{YY}}\label{rho-r}
Let $\mathcal{B} \ge 0$.
Then
$$ \min_{1 \le i \le n} r_i(\mathcal{B}) \le \rho(\mathcal{B}) \le  \max_{1 \le i \le n} r_i(\mathcal{B}). \eqno(3.1)$$
\end{lemma}

\begin{lemma}{\em \cite[Lemma 5.3]{YY}}\label{rho-s}
Let $\mathcal{B} \ge 0$, and let $x \in \mathbb{R}^n_{++}$.
Then
$$ \min_{1 \le i \le n} s_i(\mathcal{B},x) \le \rho(\mathcal{B}) \le  \max_{1 \le i \le n} s_i(\mathcal{B},x).\eqno(3.2)$$
\end{lemma}

\begin{lemma}\label{rho-eq}
Let $\mathcal{B} \ge 0$ and $x \in \mathbb{R}^n_{++}$. 
Suppose that $\mathcal{B}$ is weakly irreducible.
Then either equality in (3.1) holds if and only if $r_1(\mathcal{B})=r_2(\mathcal{B})=\cdots=r_n(\mathcal{B})$;
either equality in (3.2) holds if and only if $\mathcal{B}x^{k-1}=\rho(\mathcal{B})x^{[k-1]}$.
\end{lemma}

\begin{proof}
For completeness we restate the proof of (3.1) and (3.2) as in \cite{YY}.
We first consider the equality cases of (3.1).
Let $\alpha=\min_{1 \le i \le n} r_i$.
If $\alpha=0$, surely $\rho(\mathcal{B}) \ge \alpha=0$, and $\rho(\mathcal{B})=0$ if and only if $\mathcal{B}=0$ (see \cite[Theorem 3.1]{ZQW}).
So we assume that $\alpha>0$.
Let $\mathcal{C}$ be a tensor with the same order and dimension as $\mathcal{B}$ whose entries are defined as
$c_{i_1i_2 \ldots i_k}=\frac{\alpha}{r_{i_1}(\mathcal{B})}b_{i_1i_2 \ldots i_k}$.
Then $0 \le \mathcal{C} \le \mathcal{B}$, and by Theorem \ref{ineq1}(1), $\rho(\mathcal{B}) \ge \rho(\mathcal{C})$.
In addition, $r_i(\mathcal{C})=\alpha$ for each $i=1,2,\ldots,n$, which implies $\rho(\mathcal{C})=\alpha$ (see \cite[Lemma 5.1]{YY}).
So we get $\rho(\mathcal{B}) \ge \rho(\mathcal{C})=\alpha.$
If $\rho(\mathcal{B}) =\alpha$, then $\rho(\mathcal{B}) =\rho(\mathcal{C})$, and then $\mathcal{B} =\mathcal{C}$ by Corollary \ref{ineq2}.
This implies that $r_i(\mathcal{B})=\alpha$ for each $i=1,2,\ldots,n$, and the necessity holds.
The sufficiency is easily verified by  \cite[Lemma 5.1]{YY}.
For the right equality of (3.1), the proof is similar.

Next we consider the equality cases of (3.2).
Let $D=\hbox{diag}(x_1,x_2,\ldots,x_n)$, and let $\mathcal{E}=\mathcal{B} \cdot D^{-(k-1)} \cdot \overbrace{ D \cdots D }^{k-1}$.
Then $\mathcal{E}$ and $\mathcal{B}$ have the same eigenvalues (\cite[Theorem 2.7]{YY3}), which yields $\rho(\mathcal{B})=\rho(\mathcal{E})$.
In addition, $\mathcal{E}$ is also weakly irreducible.
Noting that $r_i(\mathcal{E})=s_i(\mathcal{B},x)$ for $i=1,2,\ldots,n$.
By (3.1), $$ \min_{1 \le i \le n}s_i(\mathcal{B},x)=\min_{1 \le i \le n} r_i(\mathcal{E}) \le \rho(\mathcal{E}) \le \max_{1 \le i \le n} r_i(\mathcal{E})
= \max_{1 \le i \le n}s_i(\mathcal{B},x). $$
If the right equality holds, then all $r_i(\mathcal{E})$, and hence all $s_i(\mathcal{B},x)$, have the same value, i.e. $\mathcal{B}x^{k-1}=\rho(\mathcal{B})x^{[k-1]}$.
The proof for the right equality is similar.
\end{proof}

\begin{corollary}\label{rho-main}
Suppose that $\mathcal{B}$ is a weakly irreducible nonnegative tensor.
If there exists a vector $y \gneq 0$ such that $\mathcal{B}y^{k-1} \lneq \mu y^{[k-1]}$ (respectively, $\mathcal{B}y^{k-1} \gneq \mu y^{[k-1]}$),
then $\rho(\mathcal{B}) < \mu$ (respectively, $\rho(\mathcal{B})> \mu$).
\end{corollary}

\begin{proof}
Assume that $\mathcal{B}y^{k-1} \lneq \mu y^{[k-1]}$.
By a similar discussion as in the proof of Theroem 1.4 (1) of \cite{CPZ}, we get $y >0$.
From the inequality and by Lemma \ref{rho-s}, $$\rho(\mathcal{B}) \le \max_{1 \le i \le n} s_i(\mathcal{B},y) \le \mu.$$
If $\rho(\mathcal{B}) = \mu$, then  $\rho(\mathcal{B}) = \max_{1 \le i \le n} s_i(\mathcal{B},y)$, which implies that $\mathcal{B}y^{k-1}=\rho(\mathcal{B})y^{[k-1]}$,
a contradiction to the assumption.

Next we assume that $\mathcal{B}y^{k-1} \gneq \mu y^{[k-1]}$.
By Theorem 5.3 of \cite{YY},
$$\rho(\mathcal{B})=\max_{x \gneq 0} \min_{x_i>0} s_i(\mathcal{B},x) \ge \mu.$$
If $\rho(\mathcal{B})= \mu$, then $\mathcal{B}y^{k-1} \gneq \rho(\mathcal{B}) y^{[k-1]}$, which implies that
$\mathcal{B}y^{k-1}= \rho(\mathcal{B}) y^{[k-1]}$ by Lemma 3.5 of \cite{YY3} as $\mathcal{B}$ is weakly irreducible; a contradiction.
\end{proof}

For the adjacency tensor of a $k$-uniform hypergraph $G$,
the eigenvector equation $\mathcal{A}(G)x^{k-1}=\lambda x^{[k-1]}$ could be interpreted as
$$ \lambda x_u^{k-1}= \sum_{\{u,u_2,u_3,\ldots, u_k\} \in E(G)} x_{u_2}x_{u_3} \cdots x_{u_k}, \mbox{~for each~} u \in V(G).\eqno(3.3)$$
The eigenvector equation $\mathcal{Q}(G)x^{k-1}=\lambda x^{[k-1]}$ could be interpreted as
$$ [\lambda-d(u)] x_u^{k-1}= \sum_{\{u,u_2,u_3,\ldots, u_k\} \in E(G)} x_{u_2}x_{u_3} \cdots x_{u_k}, \mbox{~for each~} u \in V(G).\eqno(3.4)$$

A hypergraph $G$ is {\it isomorphic} to a hypergraph $H$, if there exists a bijection $\sigma: V(G) \rightarrow V(H)$ such that $
\{v_{1},v_{2},\ldots,v_{k}\} \in E(G)$ if and only if $\{ \sigma(v_{1}),\sigma(v_{2}),\ldots,\sigma(v_{k})\} \in E(H)$.
The bijection $\sigma$ is called an {\it isomorphism} of $G$  and $H$.
If $G=H$, then $\sigma$ is called an {\it automorphism} of $G$.
Let $x$ be a vector defined on $V(G)$. Denote $x_\sigma$ to be the vector such that $(x_{\sigma})_u=x_{\sigma(u)}$ for each $u \in V(G)$.

\begin{lemma}\label{AL-EV}Let $G$ be a $k$-uniform hypergraph and $\sigma $ be an automorphism of $G$.
Let $x$ be an eigenvector of $\mathcal{A}(G)$ (respectively, $\mathcal{L}(G)$, $\mathcal{Q}(G)$) associated with an eigenvalue $\lambda$.
Then $x_{\sigma} $ is also an eigenvector of $\mathcal{A}(G)$ (respectively, $\mathcal{L}(G)$,  $\mathcal{Q}(G)$)  associated with $\lambda$.
\end{lemma}

\begin{proof}
Let $u\in V(G)$ be an arbitrary but fixed vertex. By Eq. (3.3), we have
\begin{eqnarray*}
(\mathcal{A}(G)x_{\sigma }^{k-1})_{u} &=&\sum_{\{u,u_{2},\ldots,u_{k}\} \in E(G)}(x_{\sigma})_{u_{2}}(x_{\sigma})_{u_{3}}\cdots(x_{\sigma })_{u_{k}} \\
&=&\sum_{\{u,u_{2},\ldots,u_{k}\} \in E(G) }x_{\sigma(u_{2})}x_{\sigma(u_{3})}\cdots x_{\sigma (u_{k})}\\
&=&\sum_{\{ \sigma(u),\sigma(u_{2}),\ldots,\sigma(u_{k})\} \in E(G)} x_{\sigma(u_{2})}x_{\sigma (u_{3})}\cdots x_{\sigma (u_{k})}\\
&=&\lambda x_{\sigma(u)}^{k-1} \\
&=&\lambda (x_{\sigma})_{u}^{k-1},
\end{eqnarray*}
where the fourth equality is obtained from the eigenvector equation.
Hence $x_{\sigma }$ is also an eigenvector of $\mathcal{A}(G)$  associated with the eigenvalue $\lambda$.
The proof for $\mathcal{L}(G)$ and $\mathcal{Q}(G)$ is similar by the fact $d_u=d_{\sigma(u)}$ for each $u \in V(G)$.
\end{proof}

\begin{lemma}\label{AL-PV}Let $G$ be a connected simple graph, and let $x>0$ be an eigenvector of
$\mathcal{A}(G^{k,{k \over 2}})$ (respectively, $\mathcal{Q}(G^{k,{k \over 2}})$). If $u$ and $v$ are the vertices in the same half edge of $G^{k,{k \over 2}}$, then $x_{u}=x_{v}$.
\end{lemma}

\begin{proof}
Let $\sigma$ be a permutation of $V(G^{k,{k \over 2}})$ such that it interchanges $u$ and $v$ and fix all other vertices.
It is easily seen that $\sigma$ is an automorphism of $G^{k,{k \over 2}}$.
Then by Lemma \ref{AL-EV}, $x_\sigma$ is also an eigenvector of $\mathcal{A}(G^{k,{k \over 2}})$ (respectively, $\mathcal{Q}(G^{k,{k \over 2}})$).
By Theorem \ref{PF}(2), $\mathcal{A}(G^{k,{k \over 2}})$ (respectively, $\mathcal{Q}(G^{k,{k \over 2}})$) has a unique $H^{++}$-eigenvector up to a multiple, so $x_\sigma=x$, which implies the result.
\end{proof}

Let $G,x$ be defined as in Lemma \ref{AL-PV}.
We will use $x_\v$ to denote the common value of the vertices in the half edge $\v$.

\begin{lemma} \label{eqradius}
Let $G$ be a connected simple graph, and let $x>0$ be vector defined on $V(G)$.
Let $\x>0$ be a vector defined on $V(G^{k,{k \over 2}})$ such that $\x_u=x_v^{2 \over k}$ for each vertex $u \in \v$.
Then $x$ is an eigenvector of $A(G)$ (respectively, $\mathcal{Q}(G)$) corresponding to the spectral radius $\rho$ if and only if $\x$ is an
eigenvector of $A(G^{k,{k \over 2}})$ (respectively, $\mathcal{Q}(G^{k,{k \over 2}})$) corresponding to the spectral radius $\rho$.
Hence $\rho(A(G))=\rho(\mathcal{A}(G^{k,{k \over 2}}))$ and $\rho(Q(G))=\rho(\mathcal{Q}(G^{k,{k \over 2}}))$.
\end{lemma}

\begin{proof}
Let $\x_\v$ be the common value of the vertices in $\v$ given by $\x$.
The result follows by the following equivalent equations:
$$ \rho x_u = \sum_{uv \in E(G)} x_v \Leftrightarrow \rho \x_\u^{k \over 2} = \sum_{\u\v \in E(G^{k,{k \over 2}})} \x_\v^{k \over 2}
\Leftrightarrow  \rho \x_\u^{k-1} =\sum_{\u\v \in E(G^{k,{k \over 2}})}\x_\u^{{k \over 2}-1} \x_\v^{k \over 2}, $$
$$ [\rho-d(u)] x_u = \sum_{uv \in E(G)} x_v \Leftrightarrow  [\rho-d(u)] \x_\u^{k \over 2} = \sum_{\u\v \in E(G^{k,{k \over 2}})} \x_\v^{k \over 2}
\Leftrightarrow   [\rho-d(u)] \x_\u^{k-1} =\sum_{\u\v \in E(G^{k,{k \over 2}})}\x_\u^{{k \over 2}-1} \x_\v^{k \over 2}, $$
\end{proof}

Lemma \ref{eqradius} establishes a relationship between the adjacency or signless Laplacian spectral radii of the simple graphs $G$
  and those of a class of hypergraphs $G^{k,{k \over 2}}$.
So, the results involving the adjacency or signless Laplacian spectral radii of the simple graphs hold for such kind of hypergraphs.

%
%
%

\section{Minimum spectral radius and smallest limit point}
Let $P_n,C_n$ be the (simple) path and cycle of order $n$, respectively.
Denote $\Gn$ (respectively,$\Gnb$) the class of simple connected graphs (respectively, non-bipartite graphs) of order $n$.
Denote $\Gnk=\{G^{k,{k \over 2}}: G  \in \Gn\}$ and $\Gnobk=\{G^{k,{k \over 2}}: G  \in \Gnb \}$.
By Lemma \ref{eqradius}, for a connected graph $G$, $\rho(A(G))=\rho(\mathcal{A}(G^{k,{k \over 2}}))$ and $\rho(Q(G))=\rho(\mathcal{Q}(G^{k,{k \over 2}}))$.
So, the problem of finding the hypergraphs with minimal spectral radius of the adjacency or signless Laplacian tensor among all graphs in $\Gnk$ (respectively, $\Gnobk$)
is equivalent to that of finding the simple graphs with minimal spectral radius of the adjacency or signless Laplacian matrix  among all graphs in $\Gn$ (respectively, $\Gnb$).
The results on the limit points of the adjacency or signless Laplacian spectral radii of simple graphs $G$ also hold for the hypergraphs $G^{k,{k \over 2}}$.

  Feng et.al \cite{fenglz} showed that among all graphs in $\Gnb$, the minimum adjacency spectral radius is achieved by $C_n$ for odd $n$, and by $C_{n-1}+e$ for even $n$,
  where $C_{n-1}+e$ denotes the graph obtained from $C_{n-1}$ by appending a pendant edge at some vertex.
Similar result holds for the minimum signless Laplacian spectral radius.
The proof technique is involved with the perturbation of the spectral radius of a graph after one of its edges is subdivided.

Let $G$ be a simple graph containing an edge $uw$.
Denote by $G_{u,w}$ the graph obtained from $G$ by {\it subdividing the edge} $uw$, that is, by inserting a
  new vertex say $v$ and forming two new edges $uv$ and $vw$ instead of the original edge $uw$.
An {\it internal path $P$} of $G$ is a sequence of edges $u_{1},u_{2},\ldots,u_{l}$, such that
all $u_{i}$ are distinct (except possibly $u_{1}=u_{l}$), $u_{i}u_{i+1}$ is an edge of $G$ for $i=1,2,\ldots,l-1$,
$d(u_{1})\ge 3$, $d(u_{2})=\cdots=d(u_{l-1})=2$ (unless $l=2$), and $d(u_{l})\ge 3$.
Hoffman and  Smith \cite{HS} gave the following result.

\begin{lemma}{\em \cite{HS}} \label{interpathA}
Let $G$ be a simple connected graph of order $n$.
If $uw$ is an edge of $G$ not on any internal path, and $G \ne C_n$, then $\rho(A(G_{u,w})>\rho(A(G))$.
If $uw$ is an edge of $G$  on an internal path, and $G \ne T_n$, then $\rho(A(G_{u,w})<\rho(A(G))$, where $T_n$ is obtained from $P_{n-4}$ by
appending two pendant edges at each of its two end points.
\end{lemma}

With respect to the signless Laplacian matrix of a graph, Cvetkovi\'c and Sim\'c \cite{cve}, and Feng, Li and Zhang \cite{FLZ} obtained the following similar result.

\begin{lemma}{\em \cite{cve,FLZ}}\label{interpathQ}
Let $G$ be a simple connected graph of order $n$.
If $uw$ is an edge of $G$ not on any internal path, and $G \ne C_n$, then $\rho(Q(G_{u,w})>\rho(Q(G))$.
If $uw$ is an edge of $G$  on an internal path,  then $\rho(Q(G_{u,w})<\rho(Q(G))$.
\end{lemma}

By Lemma \ref{eqradius}, combining with Lemmas \ref{interpathA} and \ref{interpathQ} we will have a parallel result for the graphs $G^{k,{k \over 2}}$.
We will call $P^{k,{k \over 2}}$ an {\it internal path of the hypergraph} $G^{k,{k \over 2}}$, where $P$ is an internal path of $G$.
We use $G^{k,{k \over 2}}_{\mathbf{u},\mathbf{w}}$ to denote the hypergraph $(G_{u,w})^{k,{k \over 2}}$.
Equivalently, $G^{k,{k \over 2}}_{\mathbf{u},\mathbf{w}}$ is obtained from $G^{k,{k \over 2}}$ by {\it subdividing the edge} $\mathbf{u}\mathbf{w}$, that is, by inserting a
  half edge say $\mathbf{v}$ and forming two new edges $\mathbf{u}\mathbf{v}$ and $\mathbf{v}\mathbf{w}$
  instead of the original edge $\mathbf{u}\mathbf{w}$.
%

\begin{corollary} \label{hyperper}
Let $G$ be a connected simple graph of order $n$.
If $uw$ is an edge of $G$ not on any internal path and $G \ne C_n$, then $\rho(\A(G^{k,{k \over 2}}))<\rho(\A(G^{k,{k \over 2}}_{\mathbf{u},\mathbf{w}}))$ and
$\rho(\Q(G^{k,{k \over 2}}))<\rho(\Q(G^{k,{k \over 2}}_{\mathbf{u},\mathbf{w}}))$.

If $uw$ is an edge of $G$ on an internal path, then $\rho(\Q(G^{k,{k \over 2}}))<\rho(\Q(G^{k,{k \over 2}}_{\mathbf{u},\mathbf{w}}))$.
If, in addition, $G \ne T_n$, then $\rho(\A(G^{k,{k \over 2}}))>\rho(\A(G^{k,{k \over 2}}_{\mathbf{u},\mathbf{w}}))$.

\end{corollary}

We can also prove Corollary \ref{hyperper} by a direct discussion following the approaches of Hoffman and Smith \cite{HS}, Cvetkovi\'c and Sim\'c \cite{cve}, and Feng, Li and Zhang \cite{FLZ}, together with the using of Corollary \ref{rho-main}.

It is known that the path $P_n$ is the unique graph with the minimum adjacency or signless Laplacian spectral radius among all graphs in $\Gn$.
So, by Lemma  \ref{eqradius}  $P_n^{k,{k \over 2}}$ is the unique one with the minimum adjacency or signless Laplacian spectral radius among all hypergraphs in $\Gnk$.
Applying Lemma  \ref{eqradius} and the result of \cite[Theorem 3.5]{fenglz} on the adjacency spectral radii of simple graphs, or using Corollaries \ref{ineq3} and \ref{hyperper}, we get the following result on the minimizing hypergraphs in $\Gnobk$.

\begin{theorem}\label{a-nob}
Among all hypergraphs in $\Gnobk$,
the minimum spectral radius of the adjacency tensor (respectively, the signless Laplacian tensor) is achieved uniquely by $C_n^{k,{k \over 2}}$ for odd $n$,
and achieved uniquely by $(C_{n-1}+e)^{k,{k \over 2}}$ for even $n$.
\end{theorem}


Hoffman \cite{Hof} observed if a simple graph $G$ properly contains a cycle, then $\rho(A(G)) > \tau^{1/2}+\tau^{-1/2}=\tau^{3/2}=\sqrt{2+\sqrt{5}}$, where $\tau=(\sqrt{5}+1)/2$ is the golden mean.
He proved that $\tau^{3/2}$ is a limit point, and found all limit points of the adjacency spectral radii less than $\tau^{3/2}$.
The work of Hoffman was extended by Shearer \cite{S} to show that every real number $r \ge \tau^{3/2}$ is the limit point of the adjacency spectral radii of simple graphs.
Furthermore, Doob \cite{Doob} proved that for each $r \ge \tau^{3/2}$ (respectively, $r \le -\tau^{3/2}$) and for any $k$, there exists a sequences of graphs whose $k$th largest eigenvalue (respectively, $k$th smallest eigenvalues) converge to $r$.
By Lemma \ref{eqradius}, we get the following result on the hypergraphs $G_n^{k,{k \over 2}}$, which correspond to the results of Hoffman \cite{Hof} and Shearer \cite{S} respectively.
Denote $ \Gk=\cup_{n \in \mathbb{N}}\Gnk$ and $\Gobk=\cup_{n \in \mathbb{N}} {^{\rm nob}}\!\mathcal{G}_n^{k,\frac{k}{2}}$.

\begin{theorem}
For $n=1,2,\ldots$, let $\beta_n$ be the positive root of $P_n(x)=x^{n+1}-(1+x+x^2+\cdots+x^{n-1})$.
Let $\alpha_n=\beta_n^{1/2}+\beta_n^{-1/2}$. Then
$2=\alpha_1 < \alpha_2 < \cdots $ are all limit points of the hypergraphs in $\Gk$ smaller than $\tau^{1/2}+\tau^{-1/2}=\lim_n \alpha_n$.
\end{theorem}

\begin{theorem}
For any $r \ge \tau^{3/2}$, there exists a sequences of hypergraphs $G_{n_t}^{k,{k \over 2}}$ whose spectral radii converge to $r$.
\end{theorem}

The smallest limit point of the adjacency spectral radii of simple graphs is $2$, which is realized by a sequence of path.
If $r < \tau^{3/2}$ is a limit point, it suffices to consider the trees by Hoffman's observation.
The construction of graphs whose adjacency spectral radii converge to $r \ge \tau^{3/2}$ in \cite{Doob,S} are trees $T(n_1,n_2,\ldots,n_k)$ called {\it caterpillars}, which is obtained from a path on vertices $v_1,v_2,\ldots,v_k$ by attaching $n_j \ge 0$ pendant edges at the vertex $v_j$ for each $j=1,2,\ldots,k$.
We could not find any known sequence of non-bipartite graphs whose adjacency spectral radii converge.
However, motivated by an example in Hoffman's work \cite{Hof}, we get the following result.

\begin{lemma}
$$\lim_{n \to \infty} \rho(A(C_{2n+1}+e))=\tau^{3/2}.$$
\end{lemma}

{\bf Proof.}
Label the vertices of $C_{2n+1}+e$ as follows:
the pendant vertex is labeled by $v_0$, starting from the vertex of degree $3$ the vertices of the cycle is labeled by $v_1, v_2, \ldots, v_{2n+1}$ clockwise.
Note that now $e=v_0v_1$.
Let $x$ be a unit Perron vector of $A(C_{2n+1}+e)$, and let $\rho:=\rho(A(C_{2n+1}+e))$.
We assert that $x_{v_1} > x_{v_2} > \cdots > x_{v_{n+1}}$.
By symmetry, $x_{v_k}=x_{v_{2n+3-k}}$ for $k=2,3,\ldots,n+1$.
By the eigenvector equation on the vertex $v_{n+1}$, noting that $\rho>\tau^{3/2}$, we get
$$ x_{v_n}=\rho x_{v_{n+1}} -x_{v_{n+2}}=\rho x_{v_{n+1}} -x_{v_{n+1}}=(\rho-1) x_{v_{n+1}}> x_{v_{n+1}}.$$
Assume that $x_{v_{n-k+1}}>x_{v_{n-k+2}}$ for $k \ge 1$.
Then $$x_{v_{n-k}}=\rho x_{v_{n-k+1}}-x_{v_{n-k+2}} > (\rho-1)x_{v_{n-k+1}}>x_{v_{n-k+1}}.$$
So we prove the assertion by induction.
Note that
$$ 1=\sum_{i=0}^{2n+1}x_{v_i}^2 > x_{v_1}^2 +2 (x_{v_2}^2+  \cdots + x_{v_{n+1}}^2) > (2n+1)x_{v_{n+1}}^2.$$
So $2 x_{v_{n+1}}^2 < \frac{2}{2n+1}$.
Noting that $x_{v_{n+1}}=x_{v_{n+2}}$, we have
\begin{align*}
\rho(A(C_{2n+1}+e)) & = \sum_{uv \in E(C_{2n+1}+e)} 2x_ux_v \\
&= x^T A(C_{2n+1}+e-v_{n+1}v_{n+2})x + 2 x_{v_{n+1}}x_{v_{n+2}} \\
& < \rho(A(C_{2n+1}+e-v_{n+1}v_{n+2}))+\frac{2}{2n+1}.
\end{align*}
As $\rho(A(C_{2n+1}+e))$ is decreasing in $n$ by Lemma \ref{interpathA} and $\rho(A(C_{2n+1}+e))> \tau^{3/2}$, the limit $\lim_{n \to \infty} \rho(A(C_{2n+1}+e))$ exists.
By the above inequality, we have
$$ \tau^{3/2} \le \lim_{n \to \infty} \rho(A(C_{2n+1}+e)) \le \lim_{n \to \infty} \rho(A(C_{2n+1}+e-v_{n+1}v_{n+2}))=\tau^{3/2},$$
where the last equality follows from Proposition 3.6 of \cite{Hof}. \hfill $\blacksquare$

By Theorem \ref{NOB} and Lemma \ref{eqradius}, we get the smallest limit point of the adjacency spectral radii of the non-odd-bipartite hypergraphs in $\Gobk$.

\begin{corollary}
The value $\tau^{3/2}$ is the smallest limit point of the spectral radii of the adjacency tensors of the non-odd-bipartite hypergraphs in $\Gobk$.
\end{corollary}

\end{document}